\documentclass[graybox,envcountsame]{svmult}

\usepackage{type1cm}        
\usepackage{makeidx}         
\usepackage[pdftex]{graphicx}        
\usepackage{multicol}        
\usepackage[bottom]{footmisc}

\usepackage{newtxtext}       %
\usepackage[varvw]{newtxmath}       

\usepackage{overpic}

\newcommand{\Thm}[1]{Theorem~\ref{th:#1}}
\newcommand{\Prop}[1]{Proposition~\ref{prop:#1}}
\newcommand{\Cor}[1]{Corollary~\ref{cor:#1}}
\newcommand{\Conj}[1]{Conjecture~\ref{conj:#1}}

\newcommand{\Eq}[1]{\eqref{eq:#1}}
\renewcommand{\a}{\alpha}\renewcommand{\b}{\beta}
\newcommand{\dl}{\delta}

\newcommand{\te}{\theta}
\newcommand{\lm}{\lambda}\newcommand{\kp}{\kappa}

\newcommand{\bD}{\mathbb{D}}
\newcommand{\bE}{\mathbb{E}}
\newcommand{\N}{\mathbb{N}}
\newcommand{\bP}{\mathbb{P}}
\newcommand{\R}{\mathbb{R}}
\newcommand{\Z}{\mathbb{Z}}
\newcommand{\cE}{\mathcal{E}}
\newcommand{\cF}{\mathcal{F}}
\newcommand{\cH}{\mathcal{H}}
\newcommand{\cP}{\mathcal{P}}

\newcommand{\ds}{d_{\mathrm{s}}}
\newcommand{\dw}{d_{\mathrm{w}}}
\newcommand{\dsloc}{d_{\mathrm{s}}^{\mathrm{loc}}}
\newcommand{\oD}{\overline\bD}

\makeindex             

\begin{document}

\title*{Estimates of the local spectral dimension of the Sierpinski gasket}
\author{Masanori Hino}
\institute{Masanori Hino \at Department of Mathematics, Kyoto University, Kyoto 606-8502, Japan, \email{hino@math.kyoto-u.ac.jp}}
%
%
\maketitle

\abstract*{We discuss quantitative estimates of the local spectral dimension of the two-dimensional Sierpinski gasket with respect to the Kusuoka measure. The present arguments were inspired by a previous study of the distribution of the Kusuoka measure by R.~Bell, C.-W.~Ho, and R.~S.~Strichartz [Energy measures of harmonic functions on the Sierpi\'nski gasket, \textit{Indiana Univ.\ Math.\ J.}\ \textbf{63} (2014), 831--868].}

\abstract{We discuss quantitative estimates of the local spectral dimension of the two-dimensional Sierpinski gasket with respect to the Kusuoka measure. The present arguments were inspired by a previous study of the distribution of the Kusuoka measure by R.~Bell, C.-W.~Ho, and R.~S.~Strichartz [Energy measures of harmonic functions on the Sierpi\'nski gasket, \textit{Indiana Univ.\ Math.\ J.}\ \textbf{63} (2014), 831--868].}

\section{Introduction}
Let us recall how to construct the two-dimensional Sierpinski gasket and the associated Dirichlet form. We take three points $p_1$, $p_2$, and $p_3$ in $\R^2$ that are the vertices of an equilateral triangle. Let $\psi_i$ $(i=1,2,3)$ be a contraction map from $\R^2$ to itself that is defined by $\psi_i(x)=(x+p_i)/2$, $x\in\R^2$. Denoted herein by $K$, the two-dimensional Sierpinski gasket is a unique nonempty compact subset of $\R^2$ such that $K=\bigcup_{i=1}^3 \psi_i(K)$.

Let $V_0=\{p_1,p_2,p_3\}$ and $V_n=\bigcup_{i=1}^3 \psi_i(V_{n-1})$ for $n\ge1$ inductively. Then, $\{V_n\}_{n=0}^\infty$ is an increasing sequence, and the closure of $V_*:=\bigcup_{n=0}^\infty V_n$ is equal to $K$. 
Let $S=\{1,2,3\}$, and $W_n=S^n$ for $n\in\Z_{\ge0}$. For each $w=w_1w_2\cdots w_n\in W_n$, we define a map $\psi_w\colon K\to K$ by $\psi_w=\psi_{w_1}\circ\cdots\circ\psi_{w_n}$ and a compact set $K_w$ by $K_w=\psi_w(K)$. Note that for $w=\emptyset\in W_0$, $\psi_w$ is defined as the identity map.
Let $W_*$ denote $\bigcup_{n\in\Z_{\ge0}}W_n$. For $w=w_1w_2\cdots w_m\in W_m$ and $w'=w'_1w'_2\cdots w'_n\in W_n$, we write $ww'$ for $w_1w_2\cdots w_m w'_1w'_2\cdots w'_n\in W_{m+n}$.

We write $p\sim q$ for distinct $p,q\in V_n$ if there exist $p',q'\in V_0$ and $w\in W_n$ such that $p=\psi_w(p')$ and $q=\psi_w(q')$.
The relation $\sim$ associates $V$ with a graph structure by setting $\{(p,q)\in V_n\times V_n\mid p\sim q\}$ as the set of edges. In general, let $l(X)$ denote the space of all real-valued functions on a countable set $X$. For $n\in\Z_{\ge0}$ and $f,g\in l(V_n)$, let
\[
Q_n(f,g)=\frac12\sum_{x,y\in V_n,\ x\sim y}(f(x)-f(y))(g(x)-g(y))
\]
and $Q_n(f)=Q_n(f,f)$. We regard $Q_n(f)$ as the total energy of the function $f$. The sequence $\{(5/3)^n Q_n(f|_{V_n})\}_{n=0}^\infty$ is proved to be nondecreasing for any function $f$ in $l(V_*)$. For each $g\in l(V_0)$, there exists a unique $f\in l(V_*)$ such that $f|_{V_0}=g$ and this sequence is a constant one. In this sense, $5/3$ is the correct scaling factor for $K$.
Let $C(K)$ denote the space of all continuous real-valued functions on $K$. For $f\in C(K)$, define $\cE(f)=\lim_{n\to\infty}(5/3)^n Q_n(f|_{V_n})\,(\le{+\infty})$ and
\[
  \cF=\{f\in C(K)\mid \cE(f)<\infty\}.
\]
For $f,g\in\cF$, let
\[
\cE(f,g)=\frac12\{\cE(f+g)-\cE(f)-\cE(g)\}.
\]
Then, for any finite Borel measure $\kp$ on $K$ with full support, $(\cE,\cF)$ is a strongly local regular Dirichlet form on $L^2(K,\kp)$. Here, $C(K)$ is identified with a subspace of $L^2(K,\kp)$.
This Dirichlet form has the following self-similarity: for $f\in\cF$ and $n\in\N$, $\psi_w^* f:=f\circ \psi_w$ belongs to $\cF$ for all $w\in W_n$ and it holds that
\begin{equation}\label{eq:ss}
\cE(f,f)=\sum_{w\in W_n}\left(\frac53\right)^n \cE(\psi_w^* f,\psi_w^* f).
\end{equation}

By invoking the general theory of Dirichlet forms, the energy measure $\nu_{f}$ of $f\in\cF$ is characterized by a unique finite Borel measure on $K$ such that
\[
 \int_K g(x)\,\nu_f(dx)=2\cE(f,f g)-\cE(f^2,g)\quad\text{for all }g\in \cF\cap C(K)
\]
(Note that the above definition is simpler than usual because $K$ is compact and $C(K)$ is continuously embedded in $L^2(K,\kp)$.) 
The measure $\nu_f$ does not have mass on any one-point sets.
From the self-similarity \Eq{ss} of $(\cE,\cF)$, it holds for all $f\in\cF$ and $n\in\N$ that
\[
\nu_f=\sum_{w\in W_n}\left(\frac53\right)^n \nu_{\psi_w^* f}.
\]
In particular, we have the following identity: for $f\in \cF$ and $w\in W_n$, 
\[
\nu_f(K_w)=2\left(\frac53\right)^n \cE(\psi_w^* f,\psi_w^* f).
\]
Unlike those on differentiable spaces, energy measures on fractals generally have no simple expressions that reveal their distributions. In this respect, Bell, Ho, and Strichartz~\cite{BHS14} studied the infinitesimal behaviors of energy measures. 
To introduce their study, we state several further notations and their properties.

For each $g\in l(V_0)$, there exists a unique $f\in\cF$ such that $f|_{V_0}=g$ and the sequence $\{(5/3)^n Q_n(f|_{V_n})\}_{n=0}^\infty$ is a constant one. Such $f$ is called harmonic, and the totality of harmonic functions will be denoted by $\cH$. This is three-dimensional as a real vector space. We can take functions $h_1$ and $h_2$ from $\cH$ such that 
\[
2\cE(h_i,h_j)=\begin{cases}
1& (i=j)\\
0& (i\ne j).
\end{cases}
\]
Define $\nu=(\nu_{h_1}+\nu_{h_2})/2$. This measure does not depend on the choice of $h_1$ and $h_2$ and is sometimes called the Kusuoka measure after Kusuoka~\cite{Ku89}.\footnote{Note that more general situations are considered in \cite{Ku89}.}
For all $f\in\cF$, $\nu_f$ is absolutely continuous with respect to $\nu$. The measure $\nu$ is singular with respect to not only the Hausdorff measure on $K$~\cite{Ku89} but also any self-similar measures on $K$~\cite{HN06}.
For $w\in W_*$, define
\[
c^{(w)}=\bigl(c_j^{(w)}\bigr)_{j\in S}=\left(\frac{\nu(K_{wj})}{\nu(K_w)}\right)_{j\in S}\in \R^3.
\]
Clearly, $c^{(w)}$ lies in the plane $H=\{{\,}^t(x_1,x_2,x_3)\in \R^3\mid x_1+x_2+x_3=1\}$. This vector describes the ratio of the distribution of $\nu|_{K_w}$ to one-step smaller similarities.
We are interested in how $\{c^{(w)}\}_{w\in W_n}$ are distributed in $H$. 
Let 
\[
\bD=\left\{{\,}^t(x_1,x_2,x_3)\in H\;\middle|\; \sum_{j=1}^3 \left(x_j-\frac13\right)^2<\frac{8}{75}\right\},
\]
and let $\oD$ (resp.\ $\partial\bD$) be defined similarly as above by replacing $<$ by $\le$ (resp.\ $=$).
Let $(r,\theta)$ be the polar coordinates of $\bD$ with center $^t(1/3,1/3,1/3)$. More specifically, $(r,\theta)\in [0,\sqrt{8/75})\times (-\pi,\pi]$ corresponds to
\[
\begin{pmatrix}1/3 \\ 1/3 \\ 1/3\end{pmatrix}
+\frac{r\cos\theta}{\sqrt{6}}\begin{pmatrix}-1\\2\\-1\end{pmatrix}
+\frac{r\sin\theta}{\sqrt{2}}\begin{pmatrix}1\\0\\-1\end{pmatrix}
\in \bD.
\]
We regard $r$ and $\theta$ as maps $\bD\to [0,\sqrt{8/75})$ and $\bD\to (-\pi,\pi]$, respectively. Here we set $\theta(1/3,1/3,1/3)=0$ by convention, which does not affect later discussions.
Bell, Ho, and Strichartz~\cite{BHS14} obtained the following result and posed conjectures.\footnote{In fact, $b^{(w)}:=\frac13+\frac54(c^{(w)}-\frac13)=\frac54 c^{(w)}-\frac1{12}$ is treated in~\cite{BHS14,Hi16} in place of $c^{(w)}$ (for this relation, see also \cite[Theorem~6.3]{BHS14}). \Thm{BHS14}, \Conj{BHS14}, \Thm{Hi16}, and \Thm{Psi} below are translations of their descriptions in terms of $c^{(w)}$.}
\begin{theorem}[{\cite[Theorem~6.5]{BHS14}, see also \cite[Theorem~3.2]{Hi16}}\label{th:BHS14}]{}
For all $w\in W_*$, $c^{(w)}\in \bD$. Moreover, $c^{(w)}$ can be arbitrarily close to $\partial\bD$.
\end{theorem}
\begin{conjecture}[see {\cite[Conjectures~7.1 and 7.2]{BHS14}}\label{conj:BHS14}]{}
Let $\lm_m$ be the uniform probability distribution on $W_m$.
\begin{enumerate}
\item The law of $r\circ c^{(w)}$ under $\lm_m$ converges to the Dirac measure at $\sqrt{8/75}$ as $m\to\infty$.
\item The law of $\theta\circ c^{(w)}$ under $\lm_m$ converges to an absolutely continuous measure on the interval $(-\pi,\pi]$.
\end{enumerate}
\end{conjecture}
Although Conjecture (ii) remains unsolved, Conjecture (i) has been solved affirmatively in a stronger sense as follows.
\begin{theorem}[see {\cite[Theorem~3.5]{Hi16}}\label{th:Hi16}]{}
Let $\kp$ be either the normalized Hausdorff measure $\lm$ on $K$ or the Kusuoka measure $\nu$ on $K$.
For $x\in K\setminus V_*$ and $m\in \N$, let $[x]_m$ denote the unique element in $W_m$ such that $x\in K_{[x]_m}$.
Then,
\[
\lim_{m\to\infty}\sum_{j=1}^3\left(c_j^{([x]_m)}-\frac13\right)^2=\frac8{75},
\quad\kp\text{-a.e.\,}x.%
\footnote{Since $\kp(V_*)=0$, it is sufficient to define $[x]_m$ for only $x\in K\setminus V_*$.}
\]
\end{theorem}
The result for $\kp=\lm$ implies Conjecture (i) because almost everywhere convergence implies convergence in law.
For $\kp=\lm$, a key to the proof is the general theory of products of random matrices (Furstenberg's theorem). For $\kp=\nu$, a key to the proof is the fact that the martingale dimension is $1$, which was first proved by Kusuoka~\cite{Ku89} for Sierpinski gaskets of arbitrary dimension; see also \cite{Hi08,Hi13} for more general fractals.

In the next section, we discuss an application of \Thm{Hi16} for $\kp=\nu$ to quantitative estimates of the local spectral dimension of the Sierpinski gasket with respect to the Kusuoka measure $\nu$.
\section{Quantitative estimates of local spectral dimension}
The transition density $p_t(x,y)$ of Brownian motion on Sierpinski gasket~$K$---which is associated with the Dirichlet form $(\cE,\cF)$ on $L^2(K,\lm)$ in our context---was extensively studied by Barlow and Perkins~\cite{BP89}. In particular, the following sub-Gaussian estimate is known:
\begin{align*}
&c_1 t^{-\ds/2}\exp\Biggl(-c_2 \Biggl(\frac{|x-y|_{\R^2}^{\dw}}{t}\Biggr)^{-1/(\dw-1)}\Biggr)
    \le 
    p_t(x,y)\\
    &\quad\le c_3 t^{-\ds/2}\exp\Biggl(-c_4 \Biggl(\frac{|x-y|_{\R^2}^{\dw}}{t}\Biggr)^{-1/(\dw-1)}\Biggr), 
    \qquad x,y\in K,\ t\in(0,1],
\end{align*}
where $c_j$ $(j=1,2,3,4)$ are positive constants, $\ds=2\log_5 3=1.36521\cdots$ is the \emph{spectral dimension}, and $\dw=\log_2 5=2.32192\cdots>2$ is the \emph{walk dimension}.
On the other hand, the transition density of the singular time-changed Brownian motion with symmetrizing measure, say $\mu$---which is associated with the Dirichlet form $(\cE,\cF)$ on $L^2(K,\mu)$---was studied in several cases.
The case when $\mu$ is a self-similar measure was studied in \cite{BK01,HKK02}, and in particular, the multifractal properties of the (local) spectral dimension and walk dimension were observed. 
The case when $\mu$ is equal to the Kusuoka measure $\nu$ was treated in \cite{MS95,Ki08,Ka12}. 
We will focus on such a case here. The behavior of the transition density $q_t(x,y)$ is somewhat Gaussian-like. Concerning the short-time asymptotics of the on-diagonal $q_t(x,x)$, in particular, the following result is known.
\begin{theorem}[{\cite[Theorem~1.3~(2) and Proposition~6.6]{Ka12}}]{}
There exists a constant $\dsloc\in (1,2\log_{25/3}5]$ such that
\[
\lim_{t\downarrow0}\frac{2\log q_t(x,x)}{-\log t}=\dsloc,
\quad \nu\text{-a.e.\,}x.
\]
Moreover, $\dsloc$ is described as
\begin{equation}\label{eq:dsloc}
\dsloc=2-\frac{2\log(5/3)}{\log(5/3)-\rho},
\end{equation}
where $\rho=\lim_{m\to\infty}\rho_m=\inf_{m\in\N} \rho_m$ with
\begin{equation}\label{eq:rhom}
\rho_m=\frac1m\sum_{w\in W_m}\nu(K_w)\log \nu(K_w).
\end{equation}
\end{theorem}
We call $\dsloc$ the \emph{local spectral dimension} of $K$ with respect to the Kusuoka measure $\nu$.
From numerical computation of $\rho_m$ with $m=16$, a quantitative estimate of $\dsloc$ is given in~\cite[Remark~6.7~(1)]{Ka12} as
\[
\left(2-\frac{2\log(5/3)}{\log(5/3)-\rho_{16}}=\right)1.27874\cdots\le \dsloc\le1.51814\cdots\bigl(=2\log_{25/3}5\bigr).
\]
It seems difficult to obtain a substantially sharper estimate of $\dsloc$ by using only the above equations \Eq{dsloc} and \Eq{rhom}.
The main object of this paper is to discuss quantitative estimates of $\dsloc$ by another approach using \Thm{Hi16} with $\kp=\nu$. 
\Thm{main}, which is stated later, provides an estimate of $\dsloc$; by using this, we will give a rigorous proof of the estimate
\begin{equation}\label{eq:130}
\begin{split}
(1.271650\cdots=)\,&\frac{15\log 3+15\log 5-14\log 7}{15\log 5-7\log 7}
\le\dsloc\\
&\le \frac{5\log 5-3\log 3}{5\log 5-4\log 3}\,(=1.300763\cdots)
\end{split}
\end{equation}
(see \Thm{main1}). 
We will also explain that numerical calculation by \emph{Mathematica}~\cite{W} suggests the estimate
\begin{equation}\label{eq:main2}
{1.291008\cdots}\le \dsloc\le {1.291026\cdots}.
\end{equation}

The first ingredient for the arguments is the following.

\begin{theorem}[see {\cite[Theorem~6.2]{BHS14}}\label{th:Psi}]{}
The correspondence $c^{(w)}\mapsto {}^t(c^{(w1)},c^{(w2)},c^{(w3)})$ for $w\in W_*$ is given by $c^{(w)}\mapsto \Psi(c^{(w)})$, where $\Psi={}^t(\Psi_1,\Psi_2,\Psi_3)\colon \bD\to\bD\times\bD\times\bD$ is defined as
\begin{align*}
&\Psi_1\begin{pmatrix}x_1\\x_2\\x_3\end{pmatrix}=\frac1{15x_1}\begin{pmatrix}10x_1\\4x_1+3x_2\\4x_1+3x_3\end{pmatrix}-\frac1{25x_1}\begin{pmatrix}1\\2\\2\end{pmatrix},\\
&\Psi_2=R^{-1}\circ\Psi_1\circ R,\quad
\Psi_3=R\circ\Psi_1\circ R^{-1},\\
&R\begin{pmatrix}x_1\\x_2\\x_3\end{pmatrix}=\begin{pmatrix}x_2\\x_3\\x_1\end{pmatrix}.
\end{align*}
\end{theorem}
Each $\Psi_j$ extends continuously to the map from $\oD$ to itself.
We remark that the restriction map $\Psi_j|_{\partial\bD}$ provides a homeomorphism from $\partial\bD$ to itself for each $j\in S$.

We define a Markov chain $\{X_m\}_{m=0}^\infty$ on $\oD$ as follows.
We set $X_0=\begin{pmatrix}1/3\\1/3\\1/3\end{pmatrix}$, and for $m\ge0$,
\[
\bP\left(X_{m+1}=\Psi_j(X_m)\;\middle|\;X_m=\begin{pmatrix}x_1\\x_2\\x_3\end{pmatrix}\right)=x_j,
\quad j\in S.
\]
\begin{proposition}\label{prop:law}
For all $m\ge0$, the law $P^{X_m}$ of $X_m$ is equal to $\sum_{w\in W_m}\nu(K_w)\dl_{c^{(w)}}$, where $\dl_z$ denotes the Dirac measure at $z$. In other words, $P^{X_m}$ coincides with the image measure of $\nu$ by the map $x\mapsto c^{([x]_m)}$, where $[x]_m$ is provided in \Thm{Hi16}.
\end{proposition}
\begin{proof}
The claim is true for $m=0$ by noting that $c^{(\emptyset)}={}^t(1/3,1/3,1/3)$ from the symmetry of the Kusuoka measure $\nu$.
Let us assume that the claim is true for $m=n$. Then, $P^{X_{n+1}}$ is equal to
\[
\sum_{w\in W_n}\nu(K_w)\Biggl(\sum_{j\in S} c_j^{(w)}\dl_{\Psi_j(c^{(w)})}\Biggr)
=\sum_{w\in W_n,\ j\in S}\nu(K_{wj})\dl_{c^{(wj)}}.
\]
Therefore, the claim is true for $m=n+1$.
\end{proof}
The Markov chain $\{X_m\}_{m=0}^\infty$ is Feller, that is, its transition operator $\cP$ defined as
\[
\cP f(x)=\sum_{j=1}^3 f(\Psi_j(x))x_j,\quad x=\begin{pmatrix}x_1\\x_2\\x_3\end{pmatrix}\in\oD,\quad f\in C(\oD)
\]
satisfies that $\cP(C(\oD))\subset C(\oD)$.

We define a function $g$ on $\oD$ by 
\begin{equation}\label{eq:g}
g\begin{pmatrix}x_1\\x_2\\x_3\end{pmatrix}=\sum_{j\in S} x_j\log x_j,
\end{equation}
where $0\log 0:=0$.
For $m\in\N$, let 
\[
\xi_m=\frac1m\sum_{k=0}^{m-1}P^{X_k}.
\]
The following proposition describes the connection between $\{X_m\}_{m=0}^\infty$ and $\rho_m$, which was introduced in \Eq{rhom}.
\begin{proposition}\label{prop:rhom}
For each $m\in\N$, 
\begin{equation}\label{eq:xim}
\rho_m=\int_{\oD}g(x)\,\xi_m(dx).
\end{equation}
\end{proposition}
\begin{proof}
From \Prop{law}, for $k\ge0$,
\begin{align*}
\bE[g(X_k)]
&=\sum_{w\in W_k}\nu(K_w)g(c^{(w)})\\
&=\sum_{w\in W_k}\nu(K_w)\sum_{j\in S}\frac{\nu(K_{wj})}{\nu(K_w)}\log\frac{\nu(K_{wj})}{\nu(K_w)}\\
&=\sum_{w\in W_k}\sum_{j\in S}\nu(K_{wj})\log\frac{\nu(K_{wj})}{\nu(K_w)}\\
&=\sum_{w'\in W_{k+1}}\nu(K_{w'})\log\nu(K_{w'})-\sum_{w\in W_k}\nu(K_{w})\log\nu(K_{w}).
\end{align*}
Therefore,
\begin{align*}
\int_{\oD}g(x)\,\xi_m(dx)
&=\frac1m\sum_{k=0}^{m-1}\bE[g(X_k)]\\
&=\frac1m\left(\sum_{w\in W_m}\nu(K_{w})\log\nu(K_{w})-\nu(K_{\emptyset})\log\nu(K_{\emptyset})\right)\\
&=\rho_m,
\end{align*}
since $\nu(K_{\emptyset})=\nu(K)=1$.
\end{proof}
Since $\oD$ is compact, there exists a subsequence $\{\xi_{m_l}\}$ of $\{\xi_m\}$ converging weakly to a probability measure $\xi$. 
By letting $m\to\infty$ along $\{m_l\}$ in \Eq{xim},
\[
\rho=\lim_{l\to\infty}\rho_{m_l}=\int_{\oD}g(x)\,\xi(dx).
\]
It is a standard fact that $\xi$ is an invariant measure. Indeed, for any $f\in C(\oD)$, by letting $l\to\infty$ in the equation
\begin{align*}
&\left|\int_{\oD}\cP f(x)\,\xi_{m_l}(dx)-\int_{\oD} f(x)\,\xi_{m_l}(dx)\right|\\
&=\left|\frac1{m_l}\sum_{k=0}^{m_l-1}\bE[f(X_{k+1})]-\frac1{m_l}\sum_{k=0}^{m_l-1}\bE[f(X_{k})]\right|\\
&=\left|\frac1{m_l}(\bE[f(X_{m_l})]-\bE[f(X_{0})])\right|\\
&\le \frac2{m_l}\sup_{x\in\oD}|f(x)|,
\end{align*}
we have 
\[
\int_{\oD}\cP f(x)\,\xi(dx)-\int_{\oD} f(x)\,\xi(dx)=0.
\]
Therefore, for all $n\in\Z_{\ge0}$,
\begin{equation}\label{eq:rho}
\rho=\int_{\oD}\cP^n g(x)\,\xi(dx).
\end{equation}
Since $P^{X_m}\circ r^{-1}$ converges to the Dirac measure at $\sqrt{8/75}$ as $m\to\infty$ from \Prop{law} and \Thm{Hi16} with $\kp=\nu$, $\xi\circ r^{-1}$ is the Dirac measure at $\sqrt{8/75}$.
That is, $\xi$ concentrates on $\partial\bD$. We can then rewrite \Eq{rho} as
\begin{equation}\label{eq:rhoPn}
\rho=\int_{\partial\bD}\cP^n g(x)\,\xi(dx).
\end{equation}
Thus, we obtain the following estimate.
\begin{theorem}\label{th:main}
For all $n\in\Z_{\ge0}$, it holds that
\begin{equation}\label{eq:estimate}
\min_{x\in\partial\bD}\cP^n g(x)\le\rho\le\max_{x\in\partial\bD}\cP^n g(x)
\end{equation}
and
\begin{equation}\label{eq:estimate2}
2-\frac{2\log(5/3)}{\log(5/3)-\max_{x\in\partial\bD}\cP^n g(x)}\le\dsloc\le 2-\frac{2\log(5/3)}{\log(5/3)-\min_{x\in\partial\bD}\cP^n g(x)}.
\end{equation}
\end{theorem}
\begin{proof}
Eq.~\Eq{estimate} follows from \Eq{rhoPn}. 
Eq.~\Eq{estimate2} follows from \Eq{estimate} and \Eq{dsloc}.
\end{proof}
\begin{remark}
Since $\cP$ is positivity-preserving on $C(\partial\bD)$ and $\cP1=1$, inequality \Eq{estimate} provides a finer estimate as $n$ increases. It is expected that $\min_{x\in\partial\bD}\cP^n g(x)$ and $\max_{x\in\partial\bD}\cP^n g(x)$ have the same limit as $n\to\infty$, but this remains to be proved.
\end{remark}
\begin{figure}[t]
\centering
\begin{tabular}{c@{\hspace{0.07\textwidth}}c}
\begin{overpic}[bb=0 0 360 229, width=0.45\textwidth,clip]{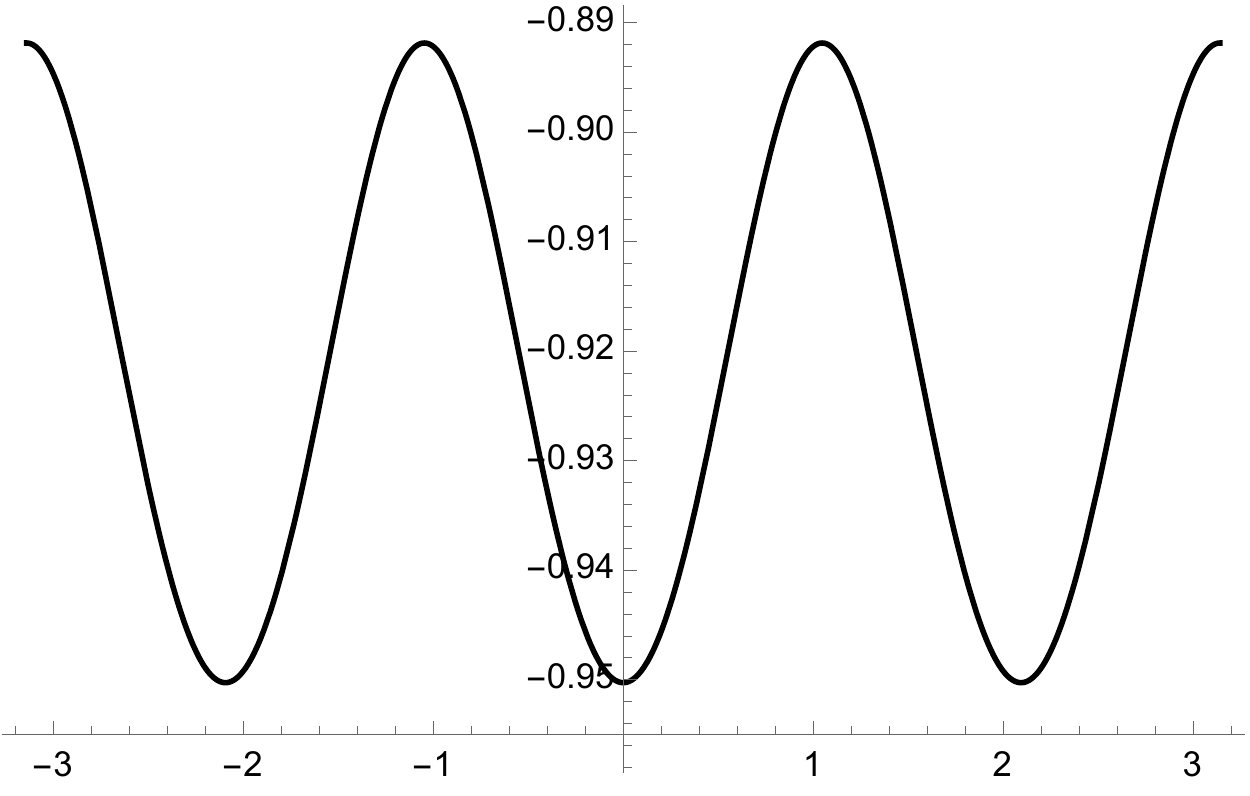}
\put(90,60){\scriptsize$g$}\end{overpic}
&\begin{overpic}[bb=0 0 360 229, width=0.45\textwidth,clip]{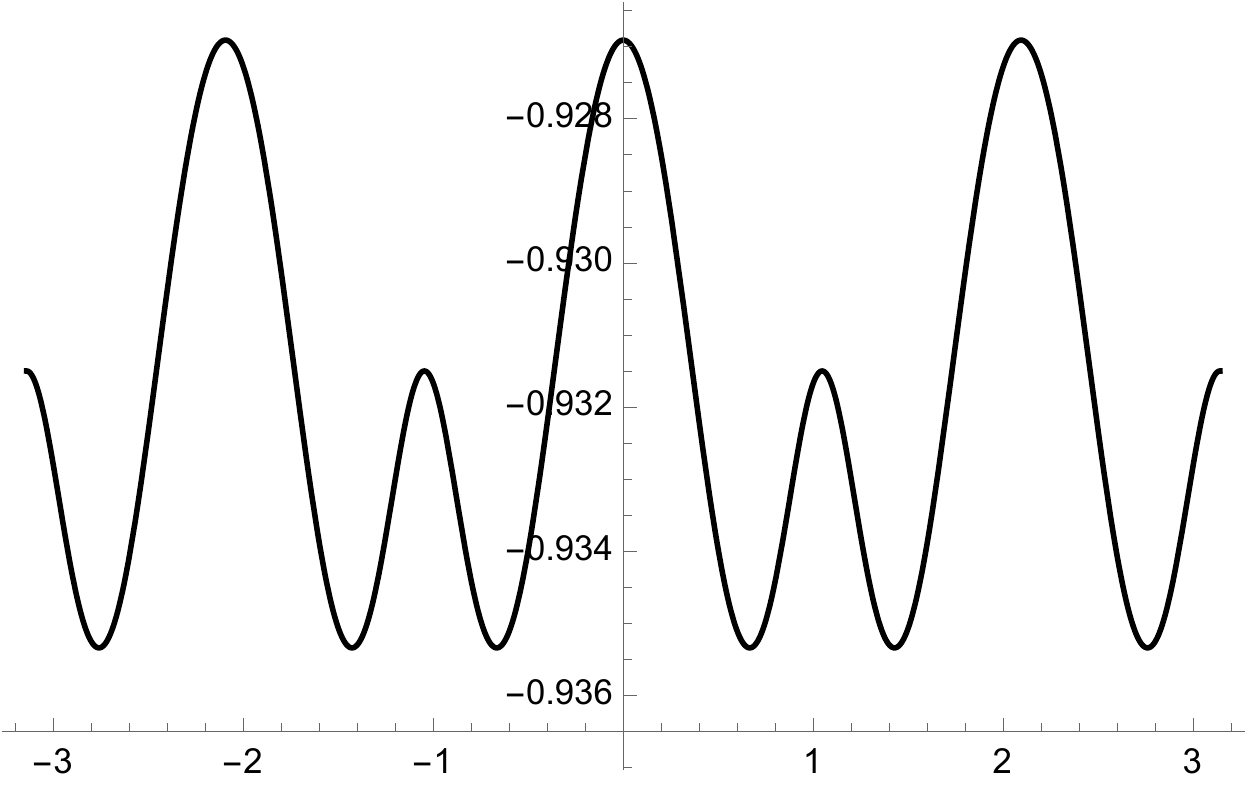}
\put(90,60){\scriptsize$\cP g$}\end{overpic}\bigskip\\
\begin{overpic}[bb=0 0 360 229, width=0.45\textwidth,clip]{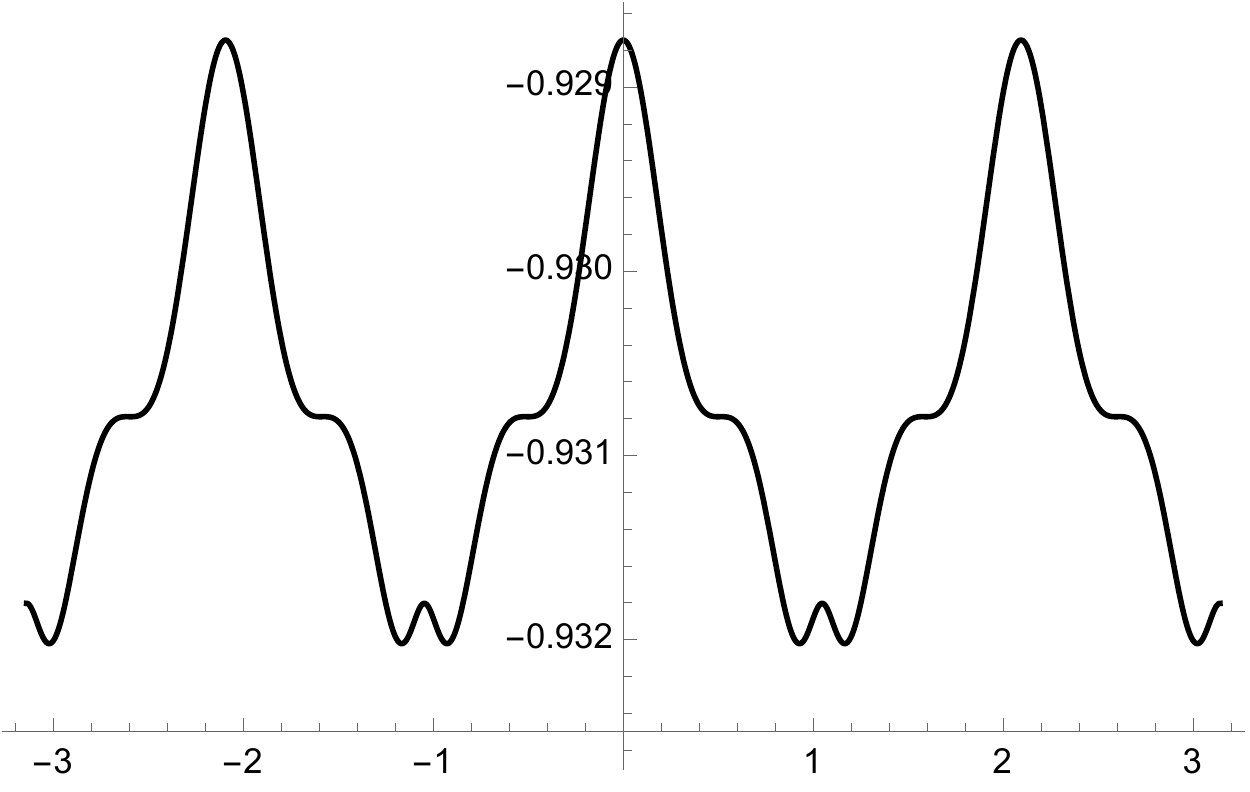}
\put(90,60){\scriptsize$\cP^2g$}\end{overpic}
&\begin{overpic}[bb=0 0 360 229, width=0.45\textwidth,clip]{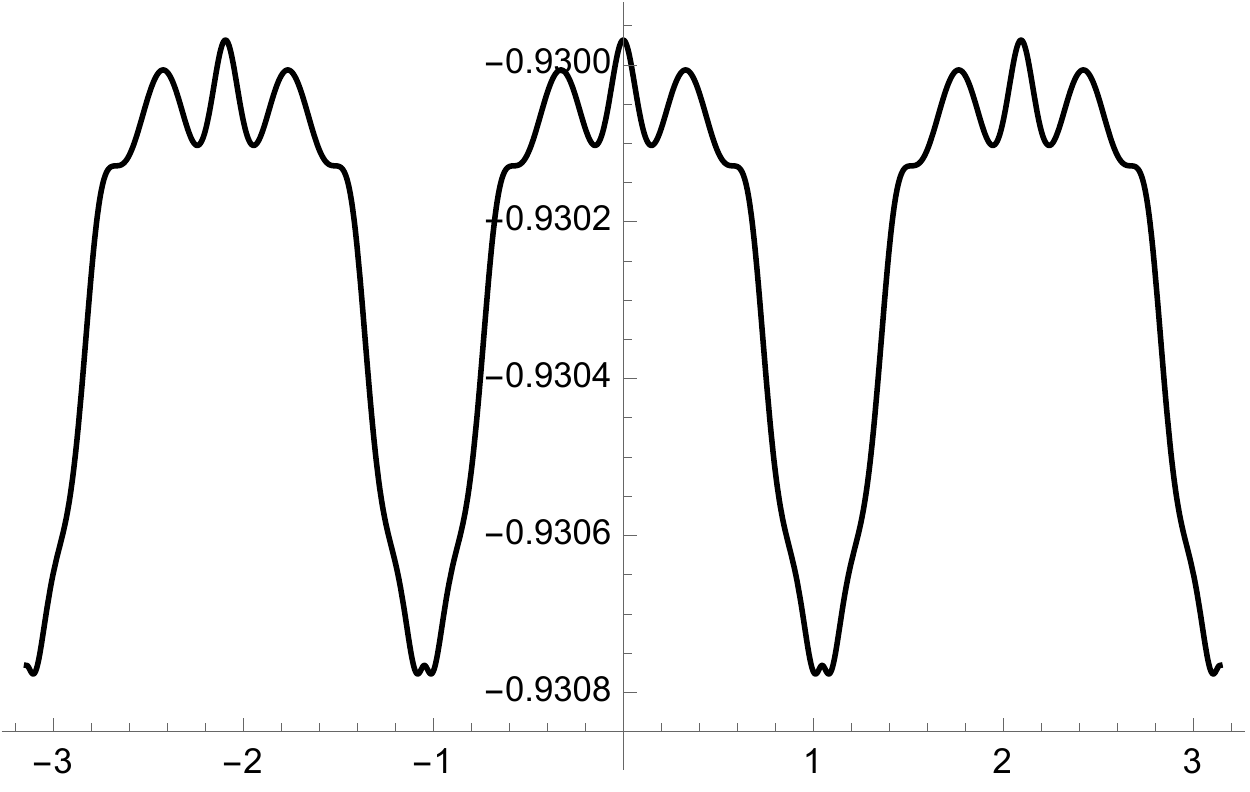}
\put(90,60){\scriptsize$\cP^3 g$}\end{overpic}\bigskip\\
\begin{overpic}[bb=0 0 360 229, width=0.45\textwidth,clip]{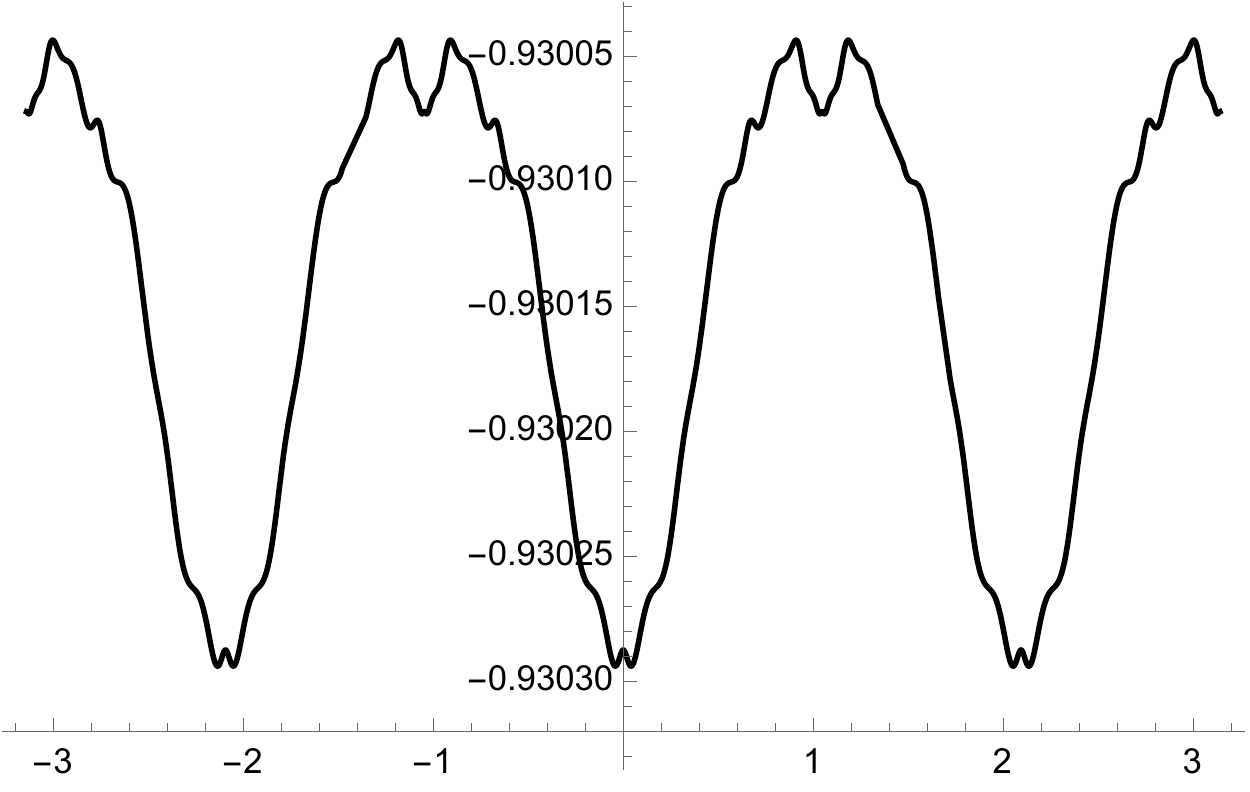}
\put(92,61){\scriptsize$\cP^4g$}\end{overpic}
&\begin{overpic}[bb=0 0 360 229, width=0.45\textwidth,clip]{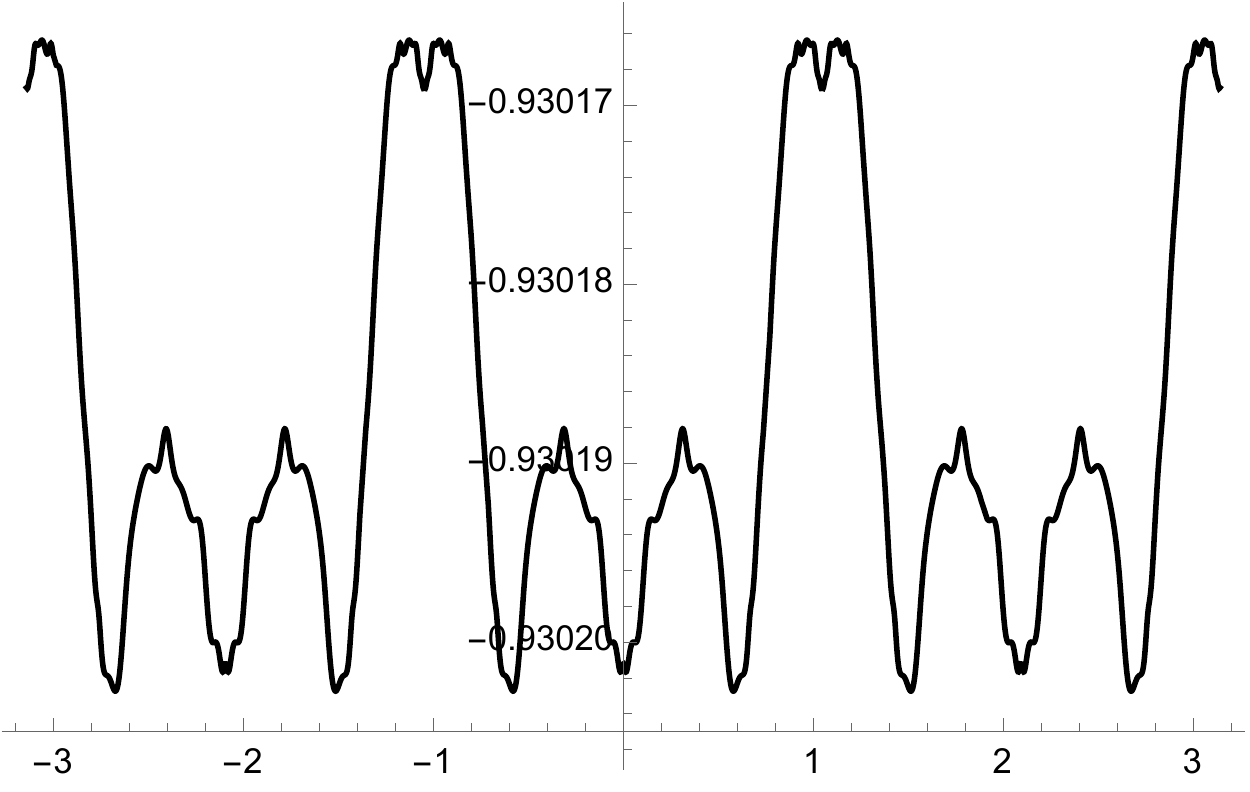}
\put(92,61){\scriptsize$\cP^5 g$}\end{overpic}\\
\end{tabular}
\caption{Graphs of $\cP^n g$, where the horizontal axis represents the argument $\theta\in(-\pi,\pi]$.}\label{fig:2}
\end{figure}
\begin{table}[t]
\centering
\caption{Upper and lower estimates of $\rho$ and $\dsloc$ based on \Thm{main}.}\label{table:1}
\begin{tabular}{c@{\qquad}c@{\qquad}c}
\hline\noalign{\smallskip}
$n$& Estimates of $\rho$ & Estimates of $\dsloc$\\
\noalign{\smallskip}\svhline\noalign{\smallskip}
$0$ & $-0.9502705\cdots\le\rho\le-0.8918673\cdots$ & $1.271650\cdots\le \dsloc\le 1.300763\cdots$\\
$1$ & $-0.9353387\cdots\le\rho\le-0.9269092\cdots$ & ${1.289402\cdots}\le \dsloc\le {1.293544\cdots}$\\
$2$ & $-0.9320224\cdots\le\rho\le-0.9287450\cdots$ & ${1.290308\cdots}\le \dsloc\le {1.291920\cdots}$\\
$3$ & $-0.9307764\cdots\le\rho\le-0.9299684\cdots$ & ${1.290911\cdots}\le \dsloc\le {1.291308\cdots}$\\
$4$ & $-0.9302937\cdots\le\rho\le-0.9300433\cdots$ & ${1.290947\cdots}\le \dsloc\le {1.291071\cdots}$\\
$5$ & $-0.9302027\cdots\le\rho\le-0.9301663\cdots$ & ${1.291008\cdots}\le \dsloc\le {1.291026\cdots}$\\
\noalign{\smallskip}\hline
\end{tabular}
\end{table}

The functions $\cP^n g$ are explicitly described in theory.
Fig.~\ref{fig:2} shows graphs of $\cP^n g$ on $\partial\bD$ for $0\le n\le 5$, where $\partial\bD$ is identified with the interval $(-\pi,\pi]$ via the map
\begin{equation}\label{eq:polar}
\phi\colon(-\pi,\pi] \ni \te\mapsto
\begin{pmatrix}1/3 \\ 1/3 \\ 1/3\end{pmatrix}
+\frac{2\cos\theta}{15}\begin{pmatrix}-1\\2\\-1\end{pmatrix}
+\frac{2\sqrt{3}\sin\theta}{15}\begin{pmatrix}1\\0\\-1\end{pmatrix}
\in \partial\bD.
\end{equation}
Table~\ref{table:1} gives the results of some numerical calculations by \textit{Mathematica}.%
\footnote{We used the command {\tt NMaxValue} to obtain the maximum and minimum of $\cP^n g$.}
According to these computations, Eq.~\Eq{main2} holds numerically; in particular, the first few digits of $\dsloc$ are $1.2910\cdots$, a value that happens to be close to $\sqrt{5/3}=1.290994\cdots$.

For reference, we provide a rigorous proof for the estimate of $\cP^0 g\,(=g)$, which implies Eq.~\Eq{130}. Even such an estimate ensures that $\dsloc$ is less than $\ds=1.36521\cdots$ (see \Cor{ds} below), which was previously unconfirmed.
\begin{theorem}\label{th:main1}
It holds that 
\begin{equation}\label{eq:estimate_g1}
\min_{x\in\partial\bD}g(x)=g(\phi(0))=\frac35\log 3-\log 5
\end{equation}
and 
\begin{equation}\label{eq:estimate_g2}
\max_{x\in\partial\bD}g(x)=g\left(\phi\left(\frac{\pi}3\right)\right)=\frac{14}{15}\log 7-\log 15.
\end{equation}
Consequently, we have
\[
2-\frac{2\log(5/3)}{\log(5/3)-g(\phi(\pi/3))}\le\dsloc\le 2-\frac{2\log(5/3)}{\log(5/3)-g(\phi(0))},
\]
that is, Eq.~\Eq{130} holds.
\end{theorem}
\begin{proof}
First, we note from \Eq{g} and \Eq{polar} that
\begin{align*}
g(\phi(\theta))
&=\left(\frac13-\frac2{15}\cos\theta+\frac{2\sqrt3}{15}\sin\theta\right)\log\left(\frac13-\frac2{15}\cos\theta+\frac{2\sqrt3}{15}\sin\theta\right)\\
&\quad+\left(\frac13+\frac{4}{15}\cos \theta\right)\log\left(\frac13+\frac{4}{15}\cos \theta\right)\\
&\quad+\left(\frac13-\frac2{15}\cos\theta-\frac{2\sqrt3}{15}\sin\theta\right)\log\left(\frac13-\frac2{15}\cos\theta-\frac{2\sqrt3}{15}\sin\theta\right).
\end{align*}
Because we can easily check the periodicity and symmetry of $g(\phi(\theta))$: 
\[
g(\phi(\theta))=g\left(\phi\left(\frac{2\pi}3+\theta\right)\right)=g\left(\phi\left(\frac{2\pi}3-\theta\right)\right),
\]
it suffices to prove that $\frac{d}{d\theta}(g(\phi(\theta)))\ge0$ for $\theta\in[0,\pi/3]$ for the validity of \Eq{estimate_g1} and \Eq{estimate_g2}.
From direct computation, we have
\begin{align*}
\frac{d}{d\theta}(g(\phi(\theta)))
&=\frac1{\sqrt3}(-x+y)\log\left(\frac13-\frac x3-y\right)
-\frac{2y}{\sqrt3}\log\left(\frac13+\frac{2x}3\right)\\
&\quad+\frac1{\sqrt3}(x+y)\log\left(\frac13-\frac x3+y\right),
\end{align*}
where
\[
x=\frac{2}5 \cos\theta
\quad\text{and}\quad 
y=\frac{2\sqrt3}{15}\sin\theta.
\]
Note that $0\le y\le 1/5\le x\le 2/5$ for $\theta\in[0,\pi/3]$.
By letting 
\[
\a=\frac{3 y}{1-x}
\quad\text{and}\quad 
\b=\frac{3(x-y)}{1-x+3y},
\]
it holds that 
\[
\frac{d}{d\theta}(g(\phi(\theta)))
=\frac1{\sqrt3}(x-y)\log\frac{1+\a}{1-\a}-\frac{2y}{\sqrt3}\log(1+\b).
\]
We now use the general inequalities
\[
\log\frac{1+\a}{1-\a}\ge2\a
\quad\text{and}\quad 
\log(1+\b)\le\b
\]
for $\a\in[0,1)$ and $\b\ge0$ to obtain that
\begin{align*}
\frac{d}{d\theta}(g(\phi(\theta)))
&\ge \frac2{\sqrt3}(x-y)\a-\frac{2y}{\sqrt3}\b\\
&=2\sqrt3 (x-y)y\left(\frac1{1-x}-\frac1{1-x+3y}\right)\\
&\ge0.
\end{align*}
Note that the last inequality becomes equality only if $y=0$ or $x=y$, that is, when $\theta=0$ or $\pi/3$. We can confirm that $\frac{d}{d\theta}(g\circ\phi)(0)=\frac{d}{d\theta}(g\circ\phi)(\pi/3)=0$, and the remaining claims follow from \Thm{main}.
\end{proof}
\begin{corollary}\label{cor:ds}
$\dsloc<\ds$.
\end{corollary}
\begin{proof}
In view of \Eq{130}, it suffices to prove
\[
\frac{5\log 5-3\log 3}{5\log 5-4\log 3}<2\log_5 3.
\]
By letting $a=\log_5 3<1$, this inequality is equivalent to $(5-3a)/(5-4a)<2a$, that is, $8a>5$.
This is equivalent to $3^8>5^5$, which is true because $3^8=6561$ and $5^5=3125$.
\end{proof}
\begin{acknowledgement}
This work was supported by JSPS KAKENHI Grant Numbers 19H00643 and 19K21833. 
\end{acknowledgement}


\begin{thebibliography}{99}
\bibitem{BK01} M. T. Barlow and T. Kumagai, Transition density asymptotics for some diffusion processes with multi-fractal structures, \textit{Electron. J. Probab.} \textbf{6} (2001), no. 9, 23 pp.
\bibitem{BP89} M. T. Barlow and E. A. Perkins, Brownian motion on the Sierpi\'nski gasket, \textit{Probab. Theory Related Fields} \textbf{79} (1988), 543--623. 
\bibitem{BHS14} R. Bell, C.-W. Ho, and R. S. Strichartz, Energy measures of harmonic functions on the Sierpi\'nski gasket, \textit{Indiana Univ. Math. J.} \textbf{63} (2014), 831--868. 
\bibitem{HKK02} B. M. Hambly, J. Kigami, and T. Kumagai, Multifractal formalisms for the local spectral and walk dimensions, \textit{Math. Proc. Cambridge Philos. Soc.} \textbf{132} (2002), 555--571. 
\bibitem{Hi08} M. Hino, Martingale dimensions for fractals, \textit{Ann. Probab.} \textbf{36} (2008), 971--991.
\bibitem{Hi13} M. Hino, Upper estimate of martingale dimension for self-similar fractals, \textit{Probab. Theory Related Fields} \textbf{156} (2013), 739--793.
\bibitem{Hi16} M. Hino, Some properties of energy measures on Sierpinski gasket type fractals, \textit{J. Fractal Geom.} \textbf{3} (2016), 245--263. 
\bibitem{HN06} M. Hino and K. Nakahara, On singularity of energy measures on self-similar sets II, \textit{Bull. Lond. Math. Soc.} \textbf{38} (2006), 1019--1032.
\bibitem{Ka12} N. Kajino, Heat kernel asymptotics for the measurable Riemannian structure on the Sierpinski gasket, \textit{Potential Anal.} \textbf{36} (2012), 67--115.
\bibitem{Ki08} J. Kigami, Measurable Riemannian geometry on the Sierpinski gasket: the Kusuoka measure and the Gaussian heat kernel estimate, \textit{Math. Ann.} \textbf{340} (2008), 781--804. \bibitem{Ku89} S. Kusuoka, Dirichlet forms on fractals and products of random matrices, \textit{Publ. Res. Inst. Math. Sci.} \textbf{25} (1989), 659--680.
\bibitem{MS95}V. Metz and K.-T. Sturm, Gaussian and non-Gaussian estimates for heat kernels on the Sierpi\'nski gasket, \textit{Dirichlet forms and stochastic processes (Beijing, 1993)}, 283--289, de Gruyter, Berlin, 1995. 
\bibitem{W} Wolfram Research, Inc., Mathematica, Ver.\ 13.0, Champaign, IL (2021).
\end{thebibliography}
\end{document}